\documentclass[pdftex,a4paper, 12pt]{article}

\RequirePackage{fullpage}
\RequirePackage[latin1]{inputenc}
\RequirePackage[T1]{fontenc}
\RequirePackage[english]{babel}
\RequirePackage{latexsym}
\RequirePackage{amssymb}
\RequirePackage{amsmath}
\RequirePackage{mathtools}
\RequirePackage{amsfonts}
\RequirePackage{amsthm}
\RequirePackage{cite}
\RequirePackage{tikz}
\RequirePackage{comment}
\RequirePackage{enumitem}

\newcommand{\R}{\mathbb{R}}

\newtheorem{theorem}{Theorem}[section]
\newtheorem{lemma}[theorem]{Lemma}
\newtheorem{corollary}[theorem]{Corollary}
\newtheorem*{definition*}{Definition}
\newtheorem{remark}[theorem]{Remark}

\newtheorem{proposition}[theorem]{Proposition}

\numberwithin{equation}{section}

\newcommand{\Sp}{\mathbb{S}}

\usepackage{cite}

\newcommand\K{\mathcal{K}^n}
\newcommand\vol{\mathrm{vol}}
\newcommand\KL{\mathrm{KL}}
\newcommand\MVal{{\bf MVal}^{SO(n)}}
\newcommand\Val{{\bf Val}}

\newcommand\costrans{\mathrm{C} \,}
\newcommand\trace{\mathrm{Tr}}

\newcommand\Kr{\mathcal{K}_{+}^2}
\newcommand\cuscite[1]{\textbf{\cite{#1}}}

\newcommand\odd[1]{#1_{\text{odd}}}
\newcommand\ev[1]{#1_{\text{ev}}}

\title{Minkowski Endomorphisms}
\author{Felix Dorrek}
\date{}

\begin{document}

\maketitle

\vspace{-0.8cm}

\begin{quote}
\footnotesize{ \vskip 1cm \noindent {\bf Abstract.}
Several open problems concerning Minkowski endomorphisms and Minkowski valuations are solved. More precisely, it is proved that all Minkowski endomorphisms are uniformly continuous but that there exist Minkowski endomorphisms that are not weakly-monotone. This answers questions posed repeatedly by Kiderlen \cuscite{Kiderlen06}, Schneider \cuscite{Schneider14} and Schuster \cuscite{Schuster07}. Furthermore, a recent representation result for Minkowski valuations by Schuster and Wannerer is improved under additional homogeneity assumptions. Also a question related to the structure of Minkowski endomorphisms by the same authors is answered. Finally, it is shown that there exists no McMullen decomposition in the class of continuous, even, $SO(n)$-equivariant and translation invariant Minkowski valuations extending a result by Parapatits and Wannerer \cuscite{ParWan13}.

  }
\end{quote}

\vspace{0.6cm}

\section{Introduction}

Let $\K$ denote the space of convex bodies (nonempty, compact, convex sets) in $\mathbb{R}^{n}$ endowed with the Hausdorff metric and Minkowski addition. Naturally, the investigation of structure preserving endomorphisms of $\K$ has attracted considerable attention (see e.g. \cuscite{ArtMil, Schneider74a, Schneider74b, Schneider74c, Schuster07, Kiderlen06}). In particular, in 1974, Schneider initiated a systematic study of continuous Minkowski-additive endomorphisms commuting with Euclidean motions. This class of endomorphisms, called Minkowski endomorphisms, turned out to be particularly interesting. 
\begin{definition*}
A Minkowski endomorphism is a continuous, $SO(n)$-equivariant and translation invariant map $\Phi: \K \to \K$ satisfying 
\[ \Phi(K + L) = \Phi(K) + \Phi(L), \qquad K, L \in \K. \]
\end{definition*}
Note that, in contrast to the original definition, we consider translation invariant instead of translation equivariant maps. However, it was pointed out by Kiderlen that these definitions lead to the same class of maps up to addition of the Steiner point map (see \cuscite{Kiderlen06} for details).
The main question of characterizing the (infinite dimensional) cone of Minkowski endomorphisms is a hard - yet interesting - one since it is intimately tied to the structure of $\K$. Schneider \cuscite{Schneider74b} established a complete classification in the case $n=2$. Since then a number of authors contributed further results and generalizations (see \cuscite{Kiderlen06, Schuster07, Schuster10, SchuWann15, SchuWann16}). In particular, Kiderlen obtained an important convolution representation. To state his result, recall that a convex body $K$ is uniquely determined by its \emph{support function} $h_K(u) = \max \{u \cdot x : \, x \in K \} $.

\begin{theorem}[\!\!\cuscite{Kiderlen06}] \label{thm:kiderlen_representation}
If $\Phi : \K \to \K$ is a Minkowski endomorphism, then there exists a unique zonal distribution $\nu \in C^{-\infty}_{o}(\Sp^{n-1})$ of order at most 2, called the generating distribution of $\Phi$,  such that 
\begin{equation} \label{eq:kiderlen}
 h_{\Phi \, K} = h_K \ast \nu  
\end{equation}
for every $K \in \K$. Moreover, $\Phi$ is uniformly continuous if and only if $\nu$ is a signed Borel measure.
\end{theorem} 
Here $C^{-\infty}_{o}(\Sp^{n-1})$ denotes the space of distributions vanishing on the restriction of linear functions to the sphere and $ h_K \ast \nu$ denotes the spherical convolution of the support function with the distribution $\nu$ (see \cuscite{Groemer}). While this theorem gives an explicit description of Minkowski endomorphisms, the important question of which distributions may occur as generating distributions remains open. In particular, it is not known whether all Minkowski endomorphisms are uniformly continuous. This was conjectured by several authors (see \cuscite{Schuster07, Kiderlen06} and \textbf{\cite[\textnormal{Chapter 3.3}]{Schneider14}}). With our first theorem, we confirm this conjecture in a slightly stronger form. For $K \in \K$, we denote the mean width of $K$ by $w(K)$.

\begin{theorem} \label{thm:main_theorem}
For every $n \geq 2$, there exists a constant $C_n \geq 0$ such that any Minkowski endomorphism $\Phi: \K \to \K$ is Lipschitz continuous with Lipschitz constant 
\[ c_{\Phi} \leq C_n \, w(\Phi B^n) .\]
\end{theorem}

As a consequence, we conclude that every Minkowski endomorphism is generated by a measure; providing a stronger form of Theorem \ref{thm:kiderlen_representation}. \\

An important class of endomorphisms that are completely characterized is that of weakly monotone Minkowski endomorphisms. We recall that the \emph{Steiner point} of a convex body $K \in \K$ is defined by $s(K) = \frac{1}{\vol(B^n)} \, \int_{\Sp^{n-1}} h_K(u) u \, du$.

\begin{definition*}
A Minkowski endomporphism $\Phi$ is called \emph{weakly monotone} if and only if it is monotone (w.r.t. set-inclusion) on the set of all convex bodies with Steiner point at the origin.
\end{definition*}

Let $\mathcal{M}_o(\Sp^{n-1})$ denote the space of all signed Borel measures on $\Sp^{n-1}$ having their center of mass at the origin. The following theorem by Kiderlen completely characterizes weakly monotone Minkowski endomorphisms.

\begin{theorem}[\!\cuscite{Kiderlen06}] \label{thm:kiderlen_positive}
Let $\Phi: \K \to \K$ be a Minkowski endomorphism. Then $\Phi$ is weakly monotone if and only if it is generated by a measure $\mu \in \mathcal{M}_o(\Sp^{n-1})$, that is the orthogonal projection of a non-negative measure $\nu \in \mathcal{M}(\Sp^{n-1})$ to $\mathcal{M}_o(\Sp^{n-1})$. Moreover, any such measure $\mu \in \mathcal{M}_o(\Sp^{n-1})$ generates a weakly-monotone Minkowski endomorphism.
\end{theorem}

Interestingly, weakly monotone Minkowski endomorphisms are the only known examples of Minkowski endomorphisms so far. Also, from Schneiders characterization for $n=2$ it follows that all endomorphisms are weakly monotone in that case. The natural question whether Minkowski endomorphisms are weakly monotone in general already implicitly appeared in \cuscite{Kiderlen06}. Later it was stressed by Schneider and Schuster (see \cuscite{Schuster07} and \textbf{\cite[\textnormal{Chapter 3.3}]{Schneider14}}). A positive answer would clearly yield a complete characterization of Minkowski endomorphisms by Theorem \ref{thm:kiderlen_positive}. However, in this article we prove the following:

\begin{theorem} \label{thm:main_theorem1}
For every $n \geq 3$, there exist Minkowski endomorphisms $\Phi: \K \to \K$ that are not weakly monotone.
\end{theorem}

More recently, the investigations of Minkowski endomorphisms were extended to Minkowski valuations which generalize the notion of Minkowski-additive maps.

\begin{definition*} A map $\phi : \K \to \mathcal{A}$ with values in an abelian semigroup $\mathcal{A}$ is a \emph{valuation} if 
\[ \phi(K) + \phi(L) = \phi(K \cup L) + \phi(K \cap L) \]
whenever $K \cup L$ is convex.
\end{definition*} 

Scalar valuations, where $\mathcal{A} = \mathbb{R}$ or $\mathbb{C}$, were probably first considered in Dehn's solution of Hilbert's third problem. As a natural and important generalization of the notion of measure they have since then played a central role in convex and discrete geometry (see \cuscite{KlainRota} and \textbf{\cite[\textnormal{Chapter 6}]{Schneider14}}). The name \emph{Minkowski valuation} for valuations with values in $\K$ was first coined by Ludwig (see \cuscite{Ludwig05}). She started a line of research focusing on Minkowski valuations that intertwine the linear group (see \cuscite{Ludwig02, Ludwig05, Ludwig10, Abardia12, AbaBer, Haberl12, SchuWann12, Wannerer11}). In most cases, it has been proven that the Minkowski valuations under consideration could be characterized as conic combinations of fundamental and well known valuations such as the projection or difference body operators. 

On the other hand, cones of Minkowski valuations that merely intertwine rotations tend to be much more diverse. As a direct generalization of Minkowski endomorphisms, we will be focusing on the cone of continuous, translation-invariant and $SO(n)$-equivariant Minkowski valuations denoted by $\MVal$. To explain why this generalizes Minkowski endomorphisms, recall that a valuation $\Phi$ is called $j$-homogeneous if $\Phi(\lambda K) = \lambda^j \Phi(K)$ for all $K \in \K$ and $\lambda \geq 0$. Let $\MVal_j$ denote the subcone of $j$-homogeneous valuations in $\MVal$. Then, by a result of Spiegel (see \cuscite{Spiegel}), the cone $\MVal_1$ is precisely the cone of Minkowski endomorphisms. 

First efforts to describe $\MVal_j$ for $j > 1$ go back to Schuster. About the same time as Kiderlen established Theorem \ref{thm:kiderlen_representation} for $\MVal_1$, Schuster obtained a similar convolution representation for the space of $(n-1)$-homogeneous valuations $\MVal_{n-1}$. Later, he was able to prove a representation result for even and smooth elements in $\MVal$ of arbitrary degree of homogeneity. In order to state his result, recall that $\vol_j(K | E)$ denotes the volume of $K \in \K$ projected to $E \in \mathrm{Gr}_{j,n}$. (For the notions of convolution on Grassmannians and smooth valuations, see \cuscite{takeuchi} and Section 2.4, respectively).

\begin{theorem}[\! \cuscite{Schuster10}]
Let $\Phi \in \MVal$ be even, smooth and homogeneous of degree $j \in \{1, \dots, n\}$. Then there exists a unique $(O(j) \times O(n-j))$-invariant and even function $f_{\Phi} \in C^{\infty}(\Sp^{n-1})$, called the Crofton function of $\Phi$, such that
\[ h_{\Phi K} = \vol_j(K | \cdot) \ast f_{\Phi}, \] for all $K \in \K$.
\end{theorem}
Note that, in the case $j = 1$, the Crofton function is equal to the generating function of $\Phi$ up to a constant. Let $\mathrm{C}: \mathcal{M}(\Sp^{n-1}) \to C(\Sp^{n-1})$ denote the cosine transform (see Section 2.2). It can be shown that
\begin{equation} \label{eq:generating_function} \costrans f = h_L , \quad L \in \K\end{equation}
is a necessary condition for a smooth function $f \in C^{\infty}(\Sp^{n-1})$ to be the Crofton measure of a $j$-homogeneous, even and smooth $\Phi \in \MVal$
(see \cuscite{SchuWann15}). When (\ref{eq:generating_function}) holds, $L$ is called a \emph{generalized zonoid} and $f$ is called the \emph{generating function of the convex body} $L$. 

From the result in \cuscite{Schuster07}, it follows that condition (\ref{eq:generating_function}) is also sufficient for a function $f \in C^{\infty}(\Sp^{n-1})$ to be the Crofton measure of an $(n-1)$-homogeneous Minkowski valuation. Schuster and Wannerer posed the problem of deciding whether generating functions of smooth convex bodies are Crofton measures for every $1 \leq j \leq n-1$ (see \cuscite{SchuWann15}). Note that, in the case $j=1$, this problem provides an inherent conjecture about the structure of the cone of smooth and even Minkowski endomorphisms. However, here we prove the following:

\begin{theorem} \label{thm:main_theorem2}
For $n \geq 3$, there exists an origin symmetric strictly convex and smooth body of revolution $L \in \mathcal{K}^n$ such that its generating function is not a generating function of an even Minkowski endomorphism.
\end{theorem} 

More recently, Schuster and Wannerer were able to obtain a general Hadwiger type theorem for $\MVal$.

\begin{theorem}[\!\cuscite{SchuWann16}] \label{thm:hadwiger}
If $\Phi: \K \to \K$ is a continuous Minkowski valuation which is translation invariant and $SO(n)$ equivariant, then there exist uniquely determined $c_0, c_n \geq 0$, $SO(n-1)$ invariant $\mu_j \in \mathcal{M}_{o}(\Sp^{n-1}),$ for $ 1 \leq j \leq n-2$, and an $SO(n-1)$ invariant $f_{n-1} \in C(\Sp^{n-1})$ such that 
\begin{equation} \label{eq:hadwiger}
h_{\Phi K} = c_0 + \sum_{j=1}^{n-2} S_j(K, \cdot) \ast \mu_j + S_{n-1}(K, \cdot) \ast f_{n-1} + c_n \vol_n(K)
\end{equation}
for every $K \in \K$.
\end{theorem}

Additionally, the authors remarked that in general the measures $\mu_j$ could not have densities in $\mathrm{L}^2 (\Sp^{n-1})$. They, however, left it as an open problem whether the $\mu_j$ are absolutely continuous with a density in $\mathrm{L}^1(\Sp^{n-1})$. Under the additional assumption of homogeneity, we are able to give a simplified  proof of Theorem \ref{thm:hadwiger} and also establish the conjectured extra regularity properties. This follows as a corollary from Theorem \ref{thm:main_theorem}.
{}
\begin{corollary} \label{corollary}
If $\Phi \in \MVal_j$, then there exists a zonal $f \in \mathrm{L}^1(\Sp^{n-1})$ such that 
\[ h_{\Phi K} = S_j(K, \cdot) \ast f \]
for every $K \in \K$.
\end{corollary}

One should mention that Corollary \ref{corollary} does not imply a stronger version of Theorem \ref{thm:hadwiger} because it is not known whether all the summands in (\ref{eq:hadwiger}) have to be support functions. The corresponding question, namely whether any element in $\MVal$ is decomposable into homogeneous Minkowski valuations, was first raised by Schneider and Schuster (see  \cuscite{SchnSchu}, Section 5 and \cuscite{Schuster10}). More generally, Parapatits and Schuster asked this question for Minkowski valuations that do not necessarily intertwine rotations (cf. \cuscite{ParSchu}).  Recently, Parapatits and Wanner proved that in this general setting such a decomposition is not possible (see \cuscite{ParWan13}). However, the original problem of whether such a decomposition exists for $\MVal$ remained open. We will introduce a novel way to construct $SO(n)$-equivariant Minkowski valuations that, together with Theorem \ref{thm:main_theorem2} and the result from \cuscite{ParWan13}, yields the following theorem.

\begin{theorem} \label{thm:main_theorem3}
If $n\geq 3$, then there exists a continuous, even, translation-invariant and $SO(n)$-equivariant Minkowski valuation $\Phi: \K \to \K$ which cannot be decomposed into a sum of homogeneous Minkowski valuations.
\end{theorem}

\vspace{0.1cm}

\section{Background Material}

In this section we collect the necessary background and definitions from convex geometry, functional analysis and analysis on the sphere. We also prove, a auxiliary corollary of a result by Weil. Finally, we recall some definitions and results from the theory of valuations that we will require later on. \\

\noindent{\bf 2.1. Cones in Locally Convex Spaces.} Let $X$ be a locally convex vector space and $X^{\ast}$ its dual space equipped with the weak-$^{\ast}$ topology. Recall that the \emph{dual space} of $X^*$ equipped with the weak-$^{\ast}$ topology can be identified with $X$ (cf. \cuscite{Rudin}). A \emph{cone} is a set $C \subseteq X$ such that $x, y \in C \Rightarrow x+y \in C$ and $x \in C \Rightarrow \lambda x \in C$ for all $\lambda \geq 0$. Given a set $M \subseteq X$ the smallest closed cone containing $M$ is denoted by 
\[ \mathrm{cone} (M) = \overline { \left\{ \sum_{i=1}^{m} \alpha_i x_i : \alpha_i \geq 0, \, m \in \mathbb{N}, \,\, x_i \in M \right\} } . \] For a cone 
$C \subseteq X $ its dual cone is defined by \[C^{\ast} := \{ f \in X^{\ast} : \, f(x) \geq 0, \, \, x \in C \} .\] 
The following theorem is a well-known consequence of the Hahn-Banach theorem. For the reader's convenience, we provide a short proof.
\begin{theorem} \label{double_dual}
Let $C \subseteq X$ be a cone. Then 
\begin{equation} \label{eq:double_dual}
C^{\ast \ast} = \overline{C} ,
\end{equation}
where the closure is taken in the topology of $X$.
\end{theorem}
\begin{proof}
First note that $C^{\ast \ast}$ is closed in the weak topology. However, since it is a convex set, by an easy argument using the Hahn-Banach theorem it is also closed in the topology of $X$. Hence we have $\overline{C} \subseteq C^{\ast \ast} $. To show equality, assume that $x \notin \overline{C}$. Recall that by the Hahn Banach seperation theorem there exists an $f \in X^{\ast}$ such that $f(x) < I = \inf \{ f(y) : y \in \overline{C} \}$. Since $0 \in \overline{C}$ clearly $I \leq 0$. Would there exist an element $y \in \overline{C}$ such that $f(y)<0$, then $z = \frac{f(x)} {f(y)} \, y \in \overline{C}$ and $f(z) = f(x)$ which is impossible. Therefore,
$f \in C^{\ast}$, since $f(x) <0$ it follows that $x \notin C^{\ast\ast}$.
\end{proof}

\noindent{\bf 2.2. Analysis on the Sphere.} 
All measures in this article are signed finite Borel measures. We denote the space of all measures on $\Sp^{n-1}$ by $\mathcal{M}(\Sp^{n-1})$. The space of measures in $\mathcal{M}(\Sp^{n-1})$ with center of mass at the origin is denoted by $\mathcal{M}_{o}(\Sp^{n-1})$. Integration over the unit sphere $\Sp^{n-1}$ in $\mathbb{R}^n$ is to be understood with respect to the $(n-1)$-dimensional Hausdorff measure. For the rest of this article let us fix a pole $\bar{e} \in \Sp^{n-1}$ on the sphere. Often it is convenient to use \emph{cylindrical coordinates} with respect to this pole. For a $j$-dimensional subspace $E \subseteq \R^n$ let us denote $\Sp^{n-1} \cap E$ by $\Sp^{j-1}(E)$. If $f \in C(\Sp^{n-1})$ and $n \geq 2$, then 
\begin{equation} \label{cylindrical_coordinates} \int_{\Sp^{n-1}} f(u) \, du = \int_{-1}^1  \int_{\Sp^{n-2}(\bar{e}^{\bot})} f\left(t\bar{e} + (1-t^2)^{\frac{1}{2}} v \right) dv  \,\, (1-t^2)^{\frac{n-3}{2}} dt. \end{equation} Two important integrals that one easily calculates using cylindrical coordinates are
\begin{equation} \label{eq:first_integral} \int_{\Sp^{n-1}} | \bar{e} \cdot u | \, du = \frac{2 \omega_{n-1}}{n-1} ,\end{equation} where $\omega_j = \mathcal{H}^{j-1} (\Sp^{j-1}) = (2 \pi^{j/2})/\Gamma(\frac{j}{2})$ and
\begin{equation} \label{eq:second_integral} \int_{\Sp^{n-1}} | \bar{e} \cdot u |^2 \, du = 2 \omega_{n-1} \frac{\sqrt{\pi}\Gamma(\frac{n+1}{2})} {4 \Gamma(\frac{n}{2} + 2)} = \frac{\omega_n}{n} . \end{equation} 
Let $\mu \in \mathcal{M}(\Sp^{n-1})$. We will denote the \emph{Radon decomposition} of $\mu$ by $\mu = \mu_{+} - \mu_{-}$. Then $\|\mu\|_{\mathrm{TV}} = \mu_{+}(\Sp^{n-1}) + \mu_{-}(\Sp^{n-1})$ is the total variation of $\mu$. We also define 
\[ \ev{\mu} = \frac{\mu + \mu^{I}}{2}, \qquad \odd{\mu} = \frac{\mu - \mu^{I}}{2},  \]
where $\mu^I(\omega) = \mu(-\omega)$ for every Borel set $\omega \subseteq \Sp^{n-1}$. Note that $\|\ev{\mu}\|_{\mathrm{TV}}, \|\odd{\mu}\|_{\mathrm{TV}} \leq \|\mu\|_{\mathrm{TV}}$. \\

The natural action of the group of rotations $SO(n)$ on $C(\Sp^{n-1})$ is given by 
\[ \theta f (u) = f(\theta^{-1} u), \qquad u \in \Sp^{n-1}, \, \theta \in SO(n) .\]
We will denote the stabilizer of $\bar{e}$ in $SO(n)$ by $SO(n-1)$.
A function $f \in C(\Sp^{n-1})$ is called \emph{zonal} if it is invariant under $SO(n-1)$. We denote the space of all zonal functions in $C(\Sp^{n-1})$ by $C(\Sp^{n-1}, \bar{e})$. If $f \in C(\Sp^{n-1})$ is a zonal function, then its \emph{associated function} $\tilde{f} \in C[-1,1]$ is defined by 
\[ \tilde{f} (t) = f(t \bar{e} + (1-t^2)^{\frac{1}{2}}v), \] for some $v \in \bar{e}^{\bot}$. It is easy to check that this does not depend on the choice of $v \in \bar{e}^{\bot}$. Conversely, given $g \in C[-1,1]$ we obtain a zonal function $f \in C(\Sp^{n-1})$ by 
\[ f(u) = g(\bar{e} \cdot u) .\] Since this operation is inverse to the construction of the associated function, we see that there is a one-one correspondence of zonal functions on the sphere and their associated functions (see \cuscite{Schuster07} for more information).

Having chosen the pole $\bar{e}$, it is possible to identify the sphere with the homogeneous space $SO(n) / SO(n-1)$. Since $SO(n)$ is a compact group, there exists a natural convolution operation on  $\mathcal{M}(SO(n))$. By the above identification this induces a convolution operation between measures on the sphere. For more details on this, see \cuscite{Groemer, takeuchi}. Here, we will directly define convolutions between measures on the sphere. For $f \in C(\Sp^{n-1})$, the $SO(n-1)$-symmetrization of $f$ is given by  \[ \bar{f} (u) := \int_{SO(n-1)} f(\theta u) \, d\theta ,\] where $d\theta$ denotes the Haar probability measure on $SO(n-1)$. Clearly, $\bar{f}$ is zonal. The \emph{convolution} of a measure $\mu \in \mathcal{M}(\Sp^{n-1})$ with a function $g \in C(\Sp^{n-1})$ is defined by 
\[ \mu \ast g \, (\theta \bar{e}) = \int_{\Sp^{n-1}} (\theta g) (u) \, d\mu(u)  .\]
Note that $\mu \ast g \in C(\Sp^{n-1})$ and that $\mu \ast g = \mu \ast \overline{g}$. It is thus sufficient to consider only zonal functions for convolutions from the right. Since left- and right-convolutions are self adjoint operations, the convolution between measures $\mu, \nu \in \mathcal{M}(\Sp^{n-1})$ can be defined by
\[ \int_{\Sp^{n-1}} g(u) \, d(\mu \ast \nu) (u) = \int_{\Sp^{n-1}} (\nu \ast g) (u) \, d\mu(u), \] for all $g \in C(\Sp^{n-1})$. Naturally, identifying a function $f \in \mathrm{L}^1(\Sp^{n-1})$ with its associated absolutely continuous measure defines the convolution of integrable functions. From the definition it follows that convolutions are associative and it is straightforward to show that the convolution of zonal functions is commutative.

An important operator on measures on the sphere is the cosine transform $\costrans : \mathcal{M}(\Sp^{n-1}) \to C(\Sp^{n-1})$. It is defined by 
\[ \costrans \mu \, (u)= \int_{\Sp^{n-1}}  |u \cdot v| \, d\mu(v) = \left( \mu \ast |\bar{e} \cdot \, \, \, | \right)  \, (u).\]
From the commutativity of the convolution of zonal functions, it immediately follows that for $\mu \in \mathcal{M}(\Sp^{n-1})$ and $g \in C(\Sp^{n-1})$ zonal
\begin{equation}
\label{eg:cosine_convolution}
\costrans \left( \mu \ast g \right) = \costrans \mu \ast g = \mu \ast \costrans g.
\end{equation}
It is well known (see for example \cuscite{Groemer}) that the cosine transform is injective on even functions.

Finally, for $f \in C(\Sp^{n-1})$ let $(f)_1 \in C(\mathbb{R}^n \setminus \{ 0 \})$ denote the $1$-homogeneous extension of $f$. The differential operator $\Box_n: C^{2}(\Sp^{n-1}) \to C(\Sp^{n-1})$ is defined by 
\[ \Box_n f = \frac{1}{n-1} \trace \left( \nabla^2 (f)_1 \right). \]
Since $\Box_n$ is $SO(n)$-equivariant it follows from standard results in harmonic analysis that for $f,g \in C^{2}(\Sp^{n-1})$ with $g$ zonal (see e.g. \cuscite{Schuster07})
\begin{equation}
\label{eq:box_convolution}
\Box_n \left( f \ast g \right) = \Box_n f \ast g = f \ast \Box_n g .
\end{equation}
Berg (cf. \cuscite{cberg}) showed that, for every $n \geq 2$, there exists a function $g_n \in \mathrm{L}^1(\Sp^{n-1})$ such that 
\begin{equation} \label{eq:berg_function}
 f = (\Box_n \, f) \ast g_n
\end{equation} for every $f \in C^{2}(\Sp^{n-1})$. Indeed, equation (\ref{eq:berg_function}) holds more generally for distributions on the sphere.

\vspace{0.2cm}

\noindent{\bf 2.3. Convex Bodies.}
In this chapter we will review fundamental facts and results from the theory of convex bodies. For a more detailed exposition confer \textbf{\cite{Schneider14}}. We will assume throughout that $n \geq 3$. Recall that $\K$ denotes the set of convex bodies (compact and convex sets) in $\mathbb{R}^n$ endowed with the Hausdorff metric and Minkowski addition. Any body $K \in \K$ is uniquely determined by its support function $h_K(u) = \max \{ u \cdot x : x \in K \}$ for $u \in \Sp^{n-1}$. Let $\mathcal{H}^{j}$ denote the $j$-dimensional Hausdorff measure. For any Borel set $\omega \subseteq \Sp^{n-1}$, the \emph{surface area measure}  of a convex body $K$ is defined by 
\[ S_{n-1}(K, \omega) = \mathcal{H}^{n-1} \{ x \in \partial K : N(K,x) \cap \omega \neq \emptyset \}, \]
where $N(K,x)$ denotes the normal cone of $K$ at the boundary point $x$. Let $B^n$ denote the Euclidean unit ball in $\mathbb{R}^n$. For every $r > 0$, the surface area measure satisfies the Steiner type formula
\[ S_{n-1}(K + r B^n , \cdot) = \sum_{j=0}^{n-1} r^{n-1-j} \binom{n-1}{j} S_j(K, \cdot) .\] The measure $S_j(K, \cdot)$ is called the \emph{area measure of order $j$} of $K$. Let us now consider more specifically convex bodies $K \in \K$  with non-empty interior and support function $h_K \in C^2(\Sp^{n-1})$. For a pair of orthogonal vectors $u$ and $v$ of unit length, the radii of curvature of such a $K$ at $u$ in direction $v$ is given by 
\[ r_K(u,v) = \frac{\partial^2}{\partial v^2 } \left(  h_K \right)_1 (u),\]
where $(f)_1$ denotes the $1$-homogeneous extension of a function $f \in C(\Sp^{n-1})$ to $\mathbb{R}^n \setminus \{ 0 \}$. The radius $r_K(u,v)$ is precisely the radius of the osculating circle to $K | \mathrm{span} \{u, v \}$ at the point $u \in \mathrm{span} \{u, v \}$. We denote the class of convex bodies with support function of class $C^2$ and everywhere positive radii of curvature by $\Kr$. A function $h \in C^{2}(\Sp^{n-1})$ is the support function of a convex body $K \in \Kr$ if and only if \begin{equation} \label{eq:convexity_condition} \frac{\partial^2}{\partial v^2 } \left(  h \right)_1 (u) > 0 \end{equation}  for all pairs of orthogonal vectors $u$ and $v$ (cf.  \textbf{\cite[\textnormal{Chapter 2.5}]{Schneider14}}). The eigenvalues of the Hessian $\nabla^2 (h_K)_1 (u)$ are the radii of curvature in the principal directions, that is, the principle radii of curvature. For $1 \leq j \leq n-1$, these are denoted by $r_j(u)$. The area measure of order $1 \leq j \leq n-1$ of a body $K \in \Kr$ is absolutely continuous with respect to the spherical Lebesgue measure. Its continuous density is given by the $j$th normalized elementary symmetric function of the principal radii of curvature: 
\[ s_j(K, \cdot) = \binom{n-1}{j}^{-1} \sum_{1 \leq i_1 < \dots < i_j \leq n-1} r_{i_1} \cdots r_{i_j} .\]
In particular 
\begin{equation} \label{eq:first_area_measure} s_1(K, \cdot) = \Box_n h_K .\end{equation} 

The general Christoffel-Minkowski problem asks for necessary and sufficient conditions
for a Borel measure on $\Sp^{n-1}$ to be the $j$-th area measure of a convex body. The answer to the special case $j=n-1$, known as Minkowski's existence theorem, is one of the fundamental theorems in the Brunn-Minkowski theory (see \textbf{\cite[\textnormal{Chapter 8.2}]{Schneider14}}). It states that a non-negative measure $\mu \in \mathcal{M}(\Sp^{n-1})$ is the surface area measure of a convex body with non empty interior if and only if $\mu$ is not concentrated on a great subsphere and has its centroid at the origin. The solution in the case $j=1$ was independently discovered by Firey \cuscite{Firey68} and Berg \cuscite{cberg}. Berg's solution essentially was to find the Green function of the $\Box_n$ operator (see (\ref{eq:berg_function})). The intermediate cases are only partially resolved and seem to be much more complicated (see \textbf{\cite[\textnormal{Chapter 8.4}]{Schneider14}}). For the special case of bodies of revolution in $\Kr$, Firey was able to give the following characterization.

\begin{theorem}[\!\cuscite{Firey70b}] \label{thm:firey_area_measures_revolution} Suppose that $1 \leq j \leq n - 1$. A zonal function $s(\bar{e}\cdot\,.\,)$ on $\mathbb{S}^{n-1}$ is the density of $S_j(K, \cdot)$ of a body of revolution $K \in \Kr$ if and only if $s$ satisfies the following conditions:
\begin{enumerate}
\item[(i)] $s$ is continuous on $(-1,1)$ and $\lim_{t\rightarrow \pm 1} s(t)$ is finite;
\item[(ii)] $\int_t^1 \xi\,s(\xi)(1-\xi^2)^{\frac{n-3}{2}}d\xi>0$ for $t \in (-1,1)$ and vanishes for $t = -1$;
\item[(iii)] $s(t)(1-t^2)^{\frac{n-1}{2}} > (n-1-j)\int_t^1 \xi\,s(\xi)(1-\xi^2)^{\frac{n-3}{2}}d\xi$ for all $t \in (-1,1)$.
\end{enumerate}
\end{theorem}

\vspace{0.2cm}

Another result by Firey that we require, concerns a concentration property of area measures. For $0 < \alpha < \frac{\pi}{2}$ let $C_{\alpha}$ denote the spherical cap given by $C_{\alpha} = \{ u \in \Sp^{n-1} : (\bar{e} \cdot u) \geq \cos \alpha \}$. 

\begin{theorem} [\!\!\cuscite{Firey70a}] \label{thm:area_measure_bound}
Let $K \in \K $ and $1 \leq j \leq n-1$. Then there exists a constant $A >0$ such that
\[ S_j(K, C_{\alpha}) \leq A \, \frac{   \sin^{n-1-j} \alpha }{\cos \alpha}  \, \|h_K\|^j .\] 
\end{theorem}

\vspace{0.8cm} 

Next we recall that an origin symmetric convex body $K \in \K$ is called a \emph{generalized zonoid} if there exists an even measure $\mu \in \mathcal{M}(\Sp^{n-1})$, called the \emph{generating measure of the convex body} $K$, such that 
\[ h_K = C \mu .\]
The next theorem, due to Weil, characterizes the cone of continuous generating functions of convex bodies.
\begin{theorem} [\!\!\cuscite{Weil82}] \label{generating_characterization}
An even function $\rho \in C(\Sp^{n-1})$ is the generating function of a smooth convex body $L$ if and only if 
\begin{equation} \label{eq:generating_characterization}
\int_{\Sp^{n-2}(w^{\bot})} (u \cdot \tilde{w})^2 \rho(u) \, du \geq 0,
\end{equation} for all $w \bot \tilde{w} \in \Sp^{n-1}$.
\end{theorem}

By introducing cylindrical coordinates we immediately get a characterization of generating functions of bodies of revolution. Let $\chi_{(a,b)}$ denote the indicator function of the interval $(a,b)$.

\begin{corollary} \label{generating_characterization_axial}
Let $\rho \in C(\Sp^{n-1})$ be even and zonal. For $0 < \alpha , \beta \leq 1$ define
\begin{align*} \psi_{\alpha, \beta} (t) &:= \chi_{(- \alpha, \alpha)}(t) \, \left(1-\frac{t^2}{\alpha^2}\right)^{\frac{n-4}{2}} \, \left(\frac{t^2}{\alpha^2} \beta^2 +  \frac{\omega_{n-1}}{n-1} \left(1-\frac{t^2}{\alpha^2}\right)(1-\beta^2) \right) .\end{align*}
Then $\rho$ is the generating function of a convex body $L$ if and only if 
\begin{equation}
\Psi_{\alpha, \beta}(\rho) := \int_{-\alpha}^{\alpha} \tilde{\rho}(t) \psi_{\alpha, \beta}(t) \, dt \geq 0,
\end{equation} for all $0 < \alpha, \beta < 1$.
\end{corollary}
\begin{proof} Clearly, if (\ref{eq:generating_characterization}) holds for all $w \neq \pm \bar{e}$ it holds in general by the continuity of $\rho$. Let therefore $w \neq \pm \bar{e}$. We introduce cylindrical coordinates on $\Sp^{n-2}(w^{\bot})$ by fixing $\bar{e}_w := \frac{\bar{e} | w^{\bot} }{ |\bar{e} | w^{\bot}|}$ as the pole. Furthermore let $\alpha = \bar{e}_w \cdot \bar{e}$ and let $\tilde{w}$ be decomposed as $\beta \bar{e}_w + \sqrt{1-\beta^2} \tilde{v}$ with $\tilde{v} \in \Sp^{n-3}(w^{\bot}\cap{\bar{e}^{\bot}})$. 
Then the integral $\int_{\Sp^{n-2}(w^{\bot})} (u \cdot \tilde{w})^2 \rho(u) \, du$ can be rewritten as \begin{align*}
\int_{-1}^1 & \int_{\Sp^{n-3}(w^{\bot}\cap{\bar{e}^{\bot}})} t^2\beta^2 +  2t \beta \sqrt{(1-t^2)(1-\beta^2)}(v \cdot \tilde{v}) + (1-t^2)(1-\beta^2) (v \cdot \tilde{v})^2 \,dv \\
 &\tilde{\rho}(\alpha t) \, (1-t^2)^{\frac{n-4}{2}} \, dt .
\end{align*} Using (\ref{eq:first_integral}), (\ref{eq:second_integral}) and the fact that $\tilde{\rho}$ is even this is further equal to 
\[ \int_{-1}^{1} \tilde{\rho}(\alpha t) \, \left( t^2 \beta^2 + \frac{\omega_{n-1}}{n-1} \, (1-t^2)(1-\beta^2) \right) (1-t^2)^{\frac{n-4}{2}} \, dt .\] Substituting $s = \alpha t$ completes the proof. 
\end{proof}

\noindent{\bf 2.4. Valuations on Convex Bodies.}
In this subsection we will review some results from the theory of valuations that we will require later on. For a more detailed background on classical valuation theory see \cuscite{KlainRota, hadwiger} and \textbf{\cite[\textnormal{Chapter 6}]{Schneider14}}. More recently, starting with groundbreaking results by Alesker \cuscite{Alesker01,Alesker03,Alesker04}, there has been enormous progress in the theory of valuations, in particular in connection to integral geometry (see e.g. \cuscite{Alesker11, BernigFu10, Bernig12, HabPar14}). 

Recall that a valuation on convex bodies is a map $\phi: \K \to \mathcal{A}$ for some abelian semi-group $\mathcal{A}$ satisfying
\[ \phi(K \cup L) + \phi(K \cap L) = \phi(K) + \phi(L), \] whenever $K, L, K \cup L \in \K$. The space of continuous and translation invariant valuations with values in $\mathbb{R}$ is denoted by $\Val$. A valuation $\phi$ is called $j$-homogeneous if $\phi(\lambda K) = \lambda^j \phi(K)$ for every $\lambda \geq 0$. The subspace of $j$-homogeneous valuations in $\Val$ is denoted by $\Val_j$. The following theorem lies at the heart of the theory of valuations.

\begin{theorem}[McMullen's decomposition]
If $\phi \in \Val$, then there exist $\phi_j \in \Val_j$, $0 \leq j \leq n$, such that 
\[ \phi = \sum_{j=0}^n \phi_j . \]
\end{theorem} It is a consequence of McMullen's decomposition that $\Val$ is a Banach space with respect to the norm of uniform convergence on bounded sets of $\K$.

A valuation $\phi$ is called \emph{even} if $\phi(-K) = \phi(K)$. An important subclass of valuations, that was introduced by Alesker, is given by the \emph{smooth} valuations.

\begin{definition*} 
A valuation $\phi \in \Val$ is called smooth if the map $A_{\phi}: GL(n) \to \Val $ given by 
\[ A_{\phi} (g) (K) = \phi( g^{-1} K), \] is smooth.
\end{definition*} It follows directly from a standard fact in representation theory, that smooth valuations form a dense subspace in $\Val$. In the remainder of this section we will consider the subspace $\Val^+_1$ of even $1$-homogeneous valuations in $\Val$. Let $\phi \in \Val^+_1$. Its \emph{Klain function} $\KL_{\phi} \in C(\Sp^{n-1})$ is the even function defined by 
\[ \KL_{\phi}(u) = \phi([0, u]). \] From Hadwiger's characterization of volume (see eg. \cuscite{hadwiger}) it follows that $\KL_{\phi}$ determines $\phi$ uniquely (see \cuscite{Klain}). The next proposition characterizes smooth $1$-homogeneous valuations. 

\begin{proposition} [\textbf{\cite[\textnormal{Appendix}]{BerParSchuWeb}}]
\label{prop:valuations}
A valuation $\phi \in \Val_{1}$ is smooth if and only if it is of the form
\[ \phi(K) = \int_{\Sp^{n-1}} h_K(u) f_{\phi}(u) \, du, \qquad K \in \K, \]
for some $f_{\phi} \in C^{\infty} (\Sp^{n-1})$.
\end{proposition} It is easy to check that for a smooth $\phi \in \Val^+_1$,
\[ \KL_{\phi} = \costrans f_{\phi} . \]

\vspace{0.1cm}

\section{Minkowski Endomorphisms}
In this section we will give the proofs of our main results regarding Minkowski endomorphisms. We recall, that by Theorem \ref{thm:kiderlen_representation}, Minkowski endomorphisms are uniquely determined by a zonal generating distribution (measure or function).

The following Lemma introduces a crucial necessary condition for generating measures of Minkowski endomorphisms.

\begin{lemma} \label{first_area_necessary}
Let $\mu \in \mathcal{M}(\Sp^{n-1})$ be the (zonal) generating measure of a Minkowski endomorphism. Then 
 \begin{equation} \label{eq:necessary_one}
 \int_{\Sp^{n-1}} s_1(K, u) \, d\mu(u) \geq 0,
 \end{equation} for all $K \in \Kr$. Moreover, if $\mu$ is absolutely continuous with continuous density $g \in C(\Sp^{n-1})$, then
 \begin{equation} \label{eq:necessary_two}
 \int_{\Sp^{n-1}} g(u) \, dS_1(K, u) \geq 0,
 \end{equation} for all $K \in \K$.
\end{lemma}
\begin{proof}
Let $\mu \in \mathcal{M}(\Sp^{n-1})$ be the generating measure of a Minkowski Endomorphism $\Phi$ and let $K \in \Kr$.  Using (\ref{eq:box_convolution}) and (\ref{eq:first_area_measure}) we compute   
\[ 0 \leq s_1(\Phi(K), \bar{e} ) = \Box_n (h_K \ast \mu) (\bar{e}) = s_1(K, \cdot ) \ast \mu (\bar{e}) = \int_{\Sp^{n-1}} s_1(K,u) \, d\mu(u).  \] 
Now let $\mu$ have a continuous density $g \in C(\Sp^{n-1})$. Then 
\[ \int_{\Sp^{n-1}} g(u) \, dS_1(K, u) \geq 0 \] for all $K \in \Kr$. Approximating an arbitrary $K \in \K$ by elements from $\Kr$ in the Hausdorff metric, we finally obtain 
\[ \int_{\Sp^{n-1}} g(u) \, dS_1(K, u) \geq 0 ,\] for all $K \in \K$ by the weak convergence of area measures.
\end{proof}

\vspace{0.1cm}

\begin{remark} \emph{
It is not to hard to see that the conditions (\ref{eq:necessary_one}) and (\ref{eq:necessary_two}) are not sufficient for a measure to be the generating measure of a Minkowski endomorphism.}
\end{remark}

We will now give the proof of Theorem \ref{thm:main_theorem}.

\begin{theorem}
For every $n \geq 2$, there exists a constant $C_n$ such that any Minkowski endomorphism $\Phi: \K \to \K$ is Lipschitz continuous with Lipschitz constant 
\[ c_{\Phi} \leq C_n \, w(\Phi B^n) .\]
\end{theorem}

\begin{proof}
Let 
\[\mathcal{S} := \{ s_1(K, \cdot) : K \in \Kr, \,  \text{ $SO(n-1)$-invariant and} \, s(K) = 0. \} \subseteq C_{o}(\Sp^{n-1}, \bar{e}) .\] By Lemma \ref{first_area_necessary}, we know that any generating measure $\mu_{\Phi}$ of a Minkowski endomorphism $\Phi$ satisfies $\mu_{\phi} \in \mathcal{S}^{\ast}$. We will therefore start by examining $\mathcal{S}^{\ast} \subseteq \mathcal{M}_{o}(\Sp^{n-1}, \bar{e})$ in more detail.
For $-1 < \alpha, \beta < 1$, consider the zonal measures $\tau_{\alpha} \in \mathcal{M}(\Sp^{n-1}, \bar{e})$ and $\sigma_{\beta} \in \mathcal{M}(\Sp^{n-1}, \bar{e})$ given by 
\[ \int_{\Sp^{n-1}} f(u) \, d\tau_{\alpha} (u) =  \int_{\alpha}^1 \tilde{f}(t) \, t(1-t^2)^{\frac{n-3}{2}} \, dt  \] and
\[ \int_{\Sp^{n-1}} f(u) \, d\sigma_{\beta}(u) = (1-\beta^2)^{\frac{n-1}{2}} \tilde{f}(\beta) - (n-2) \, \int_{\Sp^{n-1}} f(u) \, d\tau_{\beta} (u), \] for $f \in C(\Sp^{n-1}, \bar{e})$. Let now $C = \mathrm{cone} \{ \tau_{\alpha, o}, \sigma_{\beta, o} : -1 < \alpha, \beta < 1 \},$ where $\tau_{\alpha, o}$ and $\sigma_{\beta, o}$ denote the projections of $\tau_{\alpha}$ and $\sigma_{\beta}$ onto $\mathcal{M}_{o}(\Sp^{n-1}, \bar{e})$ respectively. Using Theorem \ref{thm:firey_area_measures_revolution} we immediatly see that 
\begin{equation} \label{eq:cone_one} C \subseteq S^{\ast} .\end{equation} We are now going to show that 
\begin{equation} \label{eq:cone_two} C^{\ast} \subseteq \overline{S} .\end{equation} Indeed, let $f \in C^{\ast}$ that is $\int_{\Sp^{n-1}} f(u) \, d\mu(u) \geq 0$ for all $\mu \in C$. For any $\mu$ that is a finite conic combination of the $\tau_{\alpha,o}$ and the $\sigma_{\beta,o}$ we then have 
\[ \int_{\Sp^{n-1}} (f + \epsilon)\,(u) \, d\mu(u) > 0, \] for any $\epsilon >0$ (using Theorem \ref{thm:firey_area_measures_revolution} and the fact that the constant function is the density of the first area measure of the unit ball). Using Theorem \ref{thm:firey_area_measures_revolution} again, we conclude that $f + \epsilon 1 \in S$ for every $\epsilon > 0$. Thus $f \in \overline{S}$. Now combining (\ref{eq:cone_one}), (\ref{eq:cone_two}) and applying Theorem \ref{double_dual} we obtain 
\[ S^{\ast} = C .\] 
Using this, we are now going to show that for every $n \geq 2$ there exists a constant $C_n$, such that for $\mu \in S^{\ast}$ 
\begin{equation} \label{eq:variation_inequality} \| \mu \|_{\mathrm{TV}} \leq  C_n \, \mu(\Sp^{n-1}). \end{equation}
Indeed, it is not hard to show that $\tau_{\alpha, o}$ satisfies the above equation for every $-1 < \alpha < 1$. Let $\beta \geq 0 $, then $\|(\sigma_{\beta})_+\|_{\mathrm{TV}} = (1-\beta^2)^{\frac{n-1}{2}}$ and 
\[ \|(\sigma_{\beta})_{-}\|_{\mathrm{TV}} = (n-2) \int_{\beta}^{1} t (1-t^2)^{\frac{n-3}{2}} \, dt = \frac{n-2}{n-1} \, (1-\beta^2)^{\frac{n-1}{2}}. \]
We see that (\ref{eq:variation_inequality}) holds for all $\sigma_{\beta}$ and thus also for $\sigma_{\beta, o}$. By the triangle inequality, it extends to all conic combinations of the $\tau_{\alpha, o}$ and $\sigma_{\beta, o}$. Let $(\mu_j)_{j \in \mathbb{N}}$ be a sequence of such conic combinations converging weakly to an arbitrary $\mu \in C$. Recall, that \[ \| \mu \|_{\mathrm{TV}} = \sup \, \{ \int_{\Sp^{n-1}} f(u) \, d\mu(u) : f \in C(\Sp^{n-1}), \, \| f \| =1   \} .\] Thus 
\[ \|\mu\|_{\mathrm{TV}} \leq \limsup_{j \to \infty} \|\mu_{j}\|_{\mathrm{TV}} \leq C_n \, \mu(\Sp^{n-1}) . \]
Let now $\mu$ be the generating measure of a Minkowski endomorphism. Then $\mu$ satisfies (\ref{eq:variation_inequality}) and  for any $f \in C(\Sp^{n-1})$,
\[ \|f \ast \mu\| \leq \| \mu\|_{\mathrm{TV}} \, \| f \| \leq C_n \, \mu(\Sp^{n-1}) \, \| f \| . \] Since any smooth Minkowski endomorphism has a (smooth) generating measure, we conclude that any smooth $\Phi$ is Lipschitz continuous with a Lipschitz constant $c_{\Phi} \leq C_n \, w(\Phi(B^n)).$ Let now $\Phi$ be an arbitrary Minkowski endomorphism. Then there exists a sequence $(\Phi_j)_{j \in \mathbb{N}}$ of smooth Minkowski endomorphisms that converges to $\Phi$ uniformly on compact subsets of $\K$ (cf. \textbf{\cite[\textnormal{Corollary 5.4.}]{SchuWann16}}). Hence, for every $\epsilon > 0$, there exists $j \geq 0$ such that for any compact convex sets $K, L$ we have 
\begin{align*} \| h_{\Phi K} - h_{\Phi L} \| &\leq \|h_{ \Phi K} - h_{\Phi_{j} K}\| + \|h_{\Phi_j K} - h_{\Phi_j L} \| +  \| h_{\Phi L} - h_{\Phi_j L}\| \\
&\leq C_n \, w(\Phi_j (B^n)) \, \|h_K - h_L \| + \| h_{\Phi_j K} - h_{\Phi K} \| + \| h_{\Phi L} - h_{\Phi_j L}\|  \\
&\leq C_n \, w(\Phi (B^n)) \, \|h_K - h_L\| + \epsilon .  \end{align*} We conclude that every Minkowski endomorphism $\Phi$ has a Lipschitz constant smaller or equal then $C_n \, V_1(\Phi(B^n))$. 
\end{proof}

\vspace{0.4cm}

The first step in proving Theorem \ref{thm:main_theorem1} is the following crucial Lemma.

\begin{lemma} \label{lem:rounding_endomorphism}
For any $c, C > 0$, there exists a (monotone) Minkowski endomporhism with (non-negative) generating function $g \in C(\Sp^{n-1})$ such that $g(\bar{e}) \geq C$ but \[ r_{\Phi K} (u, v) \leq c  \, \|h_K\| \] for all orthogonal pairs $u,v\in \Sp^{n-1}$ and all strictly convex and smooth bodies $K$.
\end{lemma}

\begin{proof}
Let $g \in C(\Sp^{n-1})$ be zonal, even and non-negative and let $g(\bar{e}) = C$ be its maximum. By Theorem \ref{thm:kiderlen_positive},  $g$ is the generating function of a Minkowski endomorphism $\Phi$. It remains to show that $g$ can be chosen in such a way that we also obtain the desired bound on the radii of curvature. Therefore, note that since $r_{\Phi (\theta^{-1} K)}(u,v) = r_{\Phi K} (\theta u, \theta v)$ for $\theta \in SO(n)$ it suffices to bound $r_{\Phi K}( \bar{e}, \bar{t})$ for $\bar{t} \in \Sp^{n-1}$ orthogonal to the pole $\bar{e}$ and all strictly convex and smooth bodies $K$. For these $K \in \K$, we have
\[ r_{\Phi K} (\bar{e}, \bar{t}) \leq  (n-1) \, s_1(\Phi K , \bar{e}) .\]
By (\ref{eq:box_convolution}) and (\ref{eq:first_area_measure}) we further obtain
\[s_1(\Phi K , \bar{e}) = \, \Box_n (h_K \ast g) (\bar{e}) = ( S_1(K, \cdot) \ast g)  \, (\bar{e}) = \int_{\Sp^{n-1}} g(u) \, dS_1(K, u) . \]
Let us now moreover require that the support of $g$ in the upper hemisphere is contained in the cap $C_{\alpha}$ . Then by  Theorem \ref{thm:area_measure_bound} (remember that the maximum of $g$ was chosen to be $C$)
\begin{align*}
r_{\Phi K}(\bar{e}, \bar{t}) &\leq (n-1) \, \int_{\Sp^{n-1}} g(u) \, dS_1(K, u) \\
&\leq 2(n-1)C \, S_1(K, C_{\alpha}) \\
&\leq 2(n-1)AC \, \frac{\sin^{n-2} \alpha}{\cos \alpha} \, \|h_K\|  . \end{align*}
Choosing $\alpha$ small enough completes the proof.
\end{proof}

The prove of Theorem \ref{thm:main_theorem1} now easily follows. 

\begin{theorem} \label{thm:main_theorem1}
For every $n \geq 3$, there exist non-monotone even Minkowski endomorphisms.
\end{theorem}

\begin{proof}
We are going to construct the desired endomorphism as the difference of two monotone ones. Let the first endomorphism $\Phi_1 : \K \to \K$ be given by $\Phi_1(K) = w(K) \, B^n$. Its generating function is the constant $1$ function. Observe that for all origin symmetric bodies $K$ we have 
\[ w(K) = \|h_K\|_{L^1} \geq \frac{2 \omega_{n-1}}{n-1} \|h_K\|. \]
Clearly, segments satisfy the above inequality. Since a maximal subsegment $I$ of an arbitrary origin symmetric $K \in \K$ satisfies $\| h_I \| =\|h_K\|$ but $\| h_I\|_{\mathrm{L^1}} \leq \| h_K\|_{\mathrm{L^1}}$, we see that the inequality holds in general. This now implies that
\[ r_{\Phi_1 K} (u,v) \geq  \frac{2 \omega_{n-1}}{n-1} \|h_K\|\] for all orthogonal pairs $u,v \in \Sp^{n-1}$. 
For the second endomorphism $\Phi_2$ we take any even endomorphism from Lemma \ref{lem:rounding_endomorphism} with $C > 1$ and $c < \frac{2 \omega_{n-1}}{n-1}$. Let $g$ be its generating function.
For all origin symmetric, strictly convex and smooth bodies $K$ and orthogonal pairs $u,v \in \Sp^{n-1}$, we then have
\[ \frac{\partial^2}{\partial v^2 } \left( h_{\Phi_1 K} - h_{\Phi_2 K}  \right)_1 (u) = r_{\Phi_1 K} (u,v) - r_{\Phi_2 K} (u,v) > 0. \ \]
Hence, by (\ref{eq:convexity_condition}), 
\[ h_{\Phi_1 K} - h_{\Phi_2 K} = h_K \ast (1-g) \]
is a support function for all origin symmetric, strictly convex and smooth bodies. Note that since $1-g$ is even we only need to show that origin symmetric bodies are mapped to convex bodies. Approximating an arbitrary origin symmetric $L \in \K$ by strictly convex and smooth bodies, we therefore see that $h_{\Phi K} := h_{\Phi_1 K} - h_{\Phi_2 K}$ defines a Minkowski endomorphism. Since its generating function $1-g$ attains a negative value at $\bar{e}$ it is not monotone by Theorem \ref{thm:kiderlen_positive}.
\end{proof}

\vspace{0.3cm}

We finally give the prove of Theorem \ref{thm:main_theorem2}.

\begin{theorem} \label{necessary_not_sufficient}
There exists an origin symmetric strictly convex and smooth body of revolution $L \in \mathcal{K}^n$ such that its generating function $\rho_L$ is not a generating function of an even Minkowski endomorphism.
\end{theorem}
\begin{proof}
Let $C(\Sp^{n-1}, \bar{e}) \subseteq C(\Sp^{n-1})$ denote the subspace of zonal functions. Moreover, let $\mathcal{MG}^{\infty} \subseteq C^{\infty}(\Sp^{n-1}, \bar{e})$ denote the cone of smooth generating functions of Minkowski endomorphisms and $\mathcal{G}^{\infty} \subseteq C^{\infty}(\Sp^{n-1}, \bar{e})$ denote the cone of generating functions of smooth bodies of revolution. We want to show that \[ \mathcal{G}^{\infty} \nsubseteq \mathcal{MG}^{\infty} .\] The respective cones are closed in $C^{\infty}(\Sp^{n-1}, \bar{e})$ since the cone of support functions is closed in $C(\Sp^{n-1})$. Indeed, let $g_j \in \mathcal{MG}^{\infty}$ and let $(g_j)_{j \in \mathbb{N}}$ converge to $g \in C^{\infty}(\Sp^{n-1})$. Then $h_K \ast g_j$ is a sequence of support functions converging in the Hausdorff metric for every $K \in \K$. We conclude that $g \in \mathcal{MG}^{\infty}$. An analogous argument yields that $\mathcal{G}^{\infty}$ is closed. By Theorem \ref{double_dual}, it therefore suffices to prove the relation
 \[ (\mathcal{MG}^{\infty})^{\ast} \nsubseteq (\mathcal{G}^{\infty})^{\ast} .\] Since, by Lemma \ref{first_area_necessary}, we have that $S_1(K, \cdot) \in (\mathcal{MG}^{\infty})^{\ast}$ for every body of revolution $K \in \K$, it indeed suffices to find a rotationally symmetric $K \in \K$ such that
 \[ S_1(K, \cdot) \notin (\mathcal{G}^{\infty})^{\ast}. \] 
We are going to show that the first area measure of the double cone defined by 
\[ D = \{ s \bar{e} + t v:  |s| + |t| \leq 1, \, \,  v \in \Sp^{n-2}(\bar{e}^{\bot})  \} \]
has this property. It can be shown (cf.\textbf{\cite[\textnormal{Section 3}]{GooZha98}}) that 
\[ \int_{\Sp^{n-1}} f(u) \, dS_1(D, u) = 2^{-\frac{n-5}{2}} \kappa_{n-1} \tilde{f} \left(\frac{1}{\sqrt{2}} \right) + (n-2) \int_0^{\frac{1}{\sqrt{2}}} \tilde{f}(t) (1-t^2)^{\frac{n-2}{2}} \, dt , \] for any zonal $f \in C(\Sp^{n-1})$. From Theorem \ref{double_dual} and Corollary \ref{generating_characterization_axial} it follows that 
\[ (\mathcal{G}^{\infty})^{\ast} = \mathrm{cone} \{ \Psi_{\alpha, \beta}: 0< \alpha, \beta < 1 \} \subseteq \left ( C^{\infty}(\Sp^{n-1}, \bar{e}) \right)^{\ast} ,\] where $\Psi_{\alpha, \beta}$ are the functionals defined in Corollary \ref{generating_characterization_axial}.  Let $h_{\epsilon} \in C^{\infty}(\Sp^{n-1}, \bar{e})$ be non-negative with $\tilde{h}_{\epsilon}\left(\frac{1}{\sqrt{2}}\right) =\| h_{\epsilon} \| =1$ and let $\tilde{h}_{\epsilon}$ be supported on $\left[ \frac{1}{\sqrt{2}} - \epsilon, \frac{1}{\sqrt{2}}  + \epsilon \right]$. Then there exists a constant $C$ such that
 \begin{align*} \Psi_{\alpha, \beta} (h_{\epsilon}) &\leq \int_{\frac{1}{\sqrt{2}} - \epsilon}^{\frac{1}{\sqrt{2}}  + \epsilon} \psi_{\alpha, \beta} (t) \, dt  \\
 &\leq C \, \int_{\frac{1}{\sqrt{2}} - \epsilon}^{\frac{1}{\sqrt{2}} + \epsilon} \left(1- \frac{t^2}{\alpha^2}\right)^{-\frac{1}{2}} \chi_{(-\alpha, \alpha)}(t) \, dt. \\
 &\leq C \int_{1- 2 \epsilon}^{1} (1- t^2)^{-\frac{1}{2}} \, dt.  \end{align*}
We conclude that for every $\delta > 0$, there exists $\epsilon > 0$ such that $\Psi_{\alpha, \beta}(h_{\epsilon}) \leq \delta$ for all $0 < \alpha , \beta < 1$. Moreover, it is obvious that there exists an $\epsilon > 0$ such that $\Psi_{\alpha, \beta}(h_{\epsilon}) = 0$ if $\alpha \leq \frac{1}{4}$. If $S_1(D, \cdot)$ is the weak limit of positive combinations of the $\Psi_{\alpha, \beta}$ then, since $\Psi_{\alpha,\beta} (1) \geq c$ independend of $\alpha > \frac{1}{4}$ and $\beta$, we would obtain 
\[ \int h_{\epsilon}(u) \,dS_1(D,u) \leq \tilde{c} \delta .\]
However, since 
\[ \int_{\Sp^{n-1}} h_{\epsilon}(u) \, dS_1(D, u) \geq  2^{-\frac{n-5}{2}} \kappa_{n-1} ,\] we obtain $S_1(D, \cdot) \notin (\mathcal{G}^{\infty})^{\ast}$.
\end{proof}

\vspace{0.1cm}

\section{Minkowski Valuations}

In this final section we will prove Corollary \ref{corollary} and Theorem \ref{thm:main_theorem3}. 

\begin{lemma} \label{lem:convolution_integrable}
Let $f \in \mathrm{L}^1(\Sp^{n-1})$ be zonal and $\mu \in \mathcal{M}(\Sp^{n-1})$. Then $\mu \ast f \in \mathrm{L}^1(\Sp^{n-1})$ and 
\begin{equation} \label{eq:convolution_integrable} \mu \ast f \, (\theta \bar{e}) = \int_{\Sp^{n-1}} \theta f (u) \, d\mu(u) = \int_{\Sp^{n-1}} \tilde{ f}(u \cdot \theta \bar{e}) \, d\mu(u) \end{equation}
whenever the integral on the right-hand side exists (which is the case at almost every point).
\end{lemma}
\begin{proof}
Consider the operator $f \mapsto \mu \ast f$, defined on the space of continuous functions $C(\Sp^{n-1})$. It is not hard to show that 
\[ \int_{\Sp^{n-1}} | \mu \ast f | (u) \, du \leq |\mu| \, |f|_{\mathrm{L}^1}. \] Thus, the convolution with $\mu$ is continuous on $C(\Sp^{n-1})$ in the $\mathrm{L}^1$-norm. Let now $f \in \mathrm{L}^1(\Sp^{n-1})$ and $f_i \in C(\Sp^{n-1})$ such that $f_i \to f$ in the $\mathrm{L}^1$-norm. Then $\mu \ast f_i$ converges in $\mathrm{L}^1(\Sp^{n-1})$ and, since the convergence in the $\mathrm{L}^1$-norm also implies weak convergence, we have
\[ \mu \ast f = \lim_{i \to \infty} \mu \ast f_i \in \mathrm{L}^1(\Sp^{n-1}). \]
Moreover, we know that $\mu \ast f_i$ converges point-wise for almost every $u \in \Sp^{n-1}$. Obviously the limit is given by the right-hand side of (\ref{eq:convolution_integrable}). 
\end{proof}

\begin{corollary}
Let $\Phi \in \MVal_j$. Then there exists a unique zonal $f \in \mathrm{L}^1(\Sp^{n-1})$ such that 
\[ h_{\Phi K} = S_j(K, \cdot) \ast f \]
for every $K \in \K$. Moreover, there exists a unique zonal measure $\mu \in \mathcal{M}(\Sp^{n-1})$ such that $f = \mu \ast \tilde{g_n}$.
\end{corollary}
\begin{proof}
From Theorem \ref{thm:hadwiger} it follows that 
$h_{\Phi K} = S_j(K, \cdot) \ast \nu$ for some measure $\nu \in \mathcal{M}(\Sp^{n-1})$.
Let $\Lambda: \MVal \to \MVal$ denote the derivation operator (cf. \cuscite{SchuWann15}) defined by
\[ h_{\Lambda\Phi (K)} (K) = \left.\frac{d}{dt}\right\vert_{t = 0} \, h_{\Phi (K + t B^n)} . \] It is then not too hard to show (see \cuscite{SchuWann16}), that, for $K \in \mathcal{K}^2_{+}$,
\[ h_{\Lambda^{j-1} \Phi (K)} = S_1(K, \cdot) \ast \nu = h_K \ast \Box_n \nu . \]
Here it is used that the domain of the $\Box_n$ operator and equations (\ref{eq:box_convolution}) and (\ref{eq:first_area_measure}) can be extended to distributions on the sphere (see \cuscite{SchuWann16}).
However, since $\Lambda^{j-1} \Phi \in \MVal_1$, it follows from Theorem \ref{thm:main_theorem} and (\ref{eq:berg_function}) that $\nu = \mu \ast g_n$ for some measure $\mu \in \mathcal{M}(\Sp^{n-1})$. It remains to show that $\nu = f \, du$ with $f \in \mathrm{L}^1(\Sp^{n-1})$. However, this immediately follows from Lemma \ref{lem:convolution_integrable}.

\end{proof}

The proof of Theorem \ref{thm:main_theorem3} is based on the following Proposition that introduces a new construction for $SO(n)$-equivariant Minkowski valuations.

\begin{proposition}  \label{prop:minkowski_valuations}
Let $\phi \in \mathbf{Val}$ and $L \in \K$ and let 
\begin{equation}
\Psi_{L, \phi} K \, (u) := \int_{SO(n)} \phi(\theta^{-1} K) \, h_{\theta L} (u) \, d\theta,
\end{equation}
for $u \in \Sp^{n-1}$. Then 
\begin{enumerate}[label=(\alph*)]
\item $\Psi_{L, \phi}: \K \to C(\Sp^{n-1})$ is a  continuous, translation-invariant and  $SO(n)$-equivariant valuation.
\item Let $\phi \geq 0$. Then $h_{\Phi_{L, \phi} K}  = \Psi_{L, \phi} K$ defines a continuous, translation invariant and $SO(n)$-equivariant Minkowski valuation $\Phi_{L, \phi}: \K \to \K$.
\item Let $\phi \in \Val^{+}_1$ be $SO(n-1)$-invariant and $L$ be origin symmetric. Then 
\begin{equation} \label{new_valuations_convolution}
\mathrm{C} \, (\Psi_{L, \phi} K ) =  h_K  \ast h_L \ast \mathrm{Kl}_{\phi}.
\end{equation}
\end{enumerate}

\end{proposition}
\begin{proof} 

For the $SO(n)$-equivariance of $\Psi_{L, \phi}$, let $\vartheta \in SO(n)$. Then
\begin{align*}
\Psi_{L, \phi} (\vartheta K) (u)  &= \int_{SO(n)} \phi \left( \theta^{-1} \vartheta K \right) h_{\theta L} (u) \, d\theta  .
\end{align*} By substituting $\theta = \vartheta \eta$, the right-hand side is further equal to
\begin{align*} \int_{SO(n)} \phi(\eta^{-1} K) h_{\vartheta \eta L} (u) \, d\eta &= \int_{SO(n)} \phi(\eta^{-1} K) h_{ \eta L} ( \vartheta^{-1} u) \, d\eta \\
&= \vartheta \left(\Psi_{L, \phi} K \right) (u). \end{align*}
The other properties in $(a)$ are obvious. Statement $(b)$ immediately follows from the fact that the class of support functions is a closed convex cone in $C(\Sp^{n-1})$.
For $(c)$, let  $\phi \in \Val^{+, \infty}_1$ be $SO(n-1)$-invariant. Then, since $\phi$ is smooth, by Proposition \ref{prop:valuations} there exists $f_{\phi}$ such that
\begin{align*} \phi(\theta^{-1} K) =  \int_{\Sp^{n-1}} h_{\theta^{-1}K} \, (u) \, f_{\phi}(u) \, du = \int_{\Sp^{n-1}} f_{\phi}(\theta^{-1} u) \, h_K(u) du = f_{\phi} \ast h_K \, (\theta \bar{e}) .
\end{align*}
It follows that 

\[ \Psi_{L, \phi} K =  h_K \ast f_{\phi} \ast h_L  = h_K \ast h_L \ast f_{\phi} .\] Since $\mathrm{C} f_{\phi} = \mathrm{Kl}_{\phi}$, we can finish the proof by approximation.
\end{proof}

\begin{theorem}
If $n\geq 3$, then there exists a continuous, even, translation-invariant and $SO(n)$-equivariant,  Minkowski valuation $\Phi: \K \to \K$ which cannot be decomposed into a sum of homogeneous Minkowski valuations.
\end{theorem}
\begin{proof}
Let $\varphi_{\epsilon} \in \Val_1^{+}$ be $SO(n-1)$-invariant given by  $\mathrm{Kl}\varphi_{\epsilon} = g_{\epsilon}$, where $g_{\epsilon} \in C^{\infty}(\Sp^{n-1})$ is even, non-negative, zonal and converges weakly to $\frac{1}{2} \left( \delta_{\bar{e}} + \delta_{-\bar{e}} \right) $. In \textbf{\cite[\textnormal{Lemma 5.1.}]{ParWan13}}, it was shown that there exist  constants $c_{\epsilon}, d_{\epsilon}$ such that $\phi_{\epsilon} := c_{\epsilon} + \varphi_{\epsilon} + d_{\epsilon} V_2$ is a positive valuation. Therefore, by Proposition \ref{prop:minkowski_valuations} (b), we know that $h_{\Phi_{L, \phi_{\epsilon}}K} = \Psi_{L, \phi_{\epsilon}} K$ defines an $SO(n)$-equivariant Minkowski valuation $\Phi_{L, \phi_{\epsilon}}$ for all $L \in \K$. Clearly, the $1$-homogeneous component of $\Psi_{L, \phi_{\epsilon}}$ is given by $\Psi_{L, \varphi_{\epsilon}} $.   Let us assume $\Psi_{L, \varphi_{\epsilon}} K$ is a support function for every $\epsilon >0$ and $K, L \in \K$. By (\ref{new_valuations_convolution}), we have
\[ \mathrm{C} \, (\Psi_{L, \varphi_{\epsilon}} K)  =  h_{K} \ast h_L \ast \mathrm{Kl}\varphi_{\epsilon}. \] Thus,
\[ \mathrm{C}^{-1} h_K \ast h_L \] has to be a support function for all convex bodies $K$ and $L$. In particular, this implies that 
\[ h_K \ast \rho_L \] is a support function for all $K \in \K$ and generalized zonoids $L \in \K$. Consequently, $\rho_L$ would have to be the generating measure of a Minkowski endomorphism for every generalized zonoid $L \in \K$. By Theorem \ref{necessary_not_sufficient} this cannot be true. 

\end{proof}

\noindent {{\bf Acknowledgments} The work of the author was
supported by the European Research Council (ERC), Project number: 306445, and the Austrian Science Fund (FWF), Project number:
Y603-N26.

\noindent Vienna University of Technology \par \noindent Institute of Discrete
Mathematics and Geometry \par \noindent Wiedner Hauptstra\ss e 8--10/1047
\par \noindent A--1040 Vienna, Austria

\vspace{0.2cm}

\par \noindent felix.dorrek@tuwien.ac.at
\end{document}